\documentclass[12pt]{amsart}
\usepackage[latin1]{inputenc}
\usepackage[english]{babel}
\usepackage{amsmath,amssymb,amsthm,amsfonts, color}
\usepackage{enumerate}
\usepackage{graphicx}
\usepackage{latexsym}
\usepackage[all]{xy}
\usepackage{tikz}

\makeatletter
\@namedef{subjclassname@2010}{ \textup{2010} Mathematics Subject Classification}
\makeatother

\newtheorem{theorem}{Theorem}
\newtheorem{corollary}{Corollary}
\newtheorem{propos}{Proposition}
\newtheorem{remark}{Remark}
\newtheorem{defin}{Definition}

\theoremstyle{definition}

\DeclareMathOperator{\supp}{supp}

\def\fg{{\mathbb N}}

\begin{document}
\title{Duality problem for disjointly homogeneous rearrangement invariant spaces}
\thanks{{\rm *}The author has been supported by the Ministry of Education and Science 
of the Russian Federation (project 1.470.2016/1.4) and by the RFBR grant 18--01--00414.}
\author[Astashkin]{Sergey V. Astashkin}
\address[Sergey V. Astashkin]{Department of Mathematics, Samara National Research University, Moskovskoye shosse 34, 443086, Samara, Russia
}
\email{\texttt{astash56@mail.ru}}
\maketitle

\vspace{-7mm}

\begin{abstract}
Let $1\le p<\infty$. A Banach lattice $E$ is said to be disjointly homogeneous (resp.  $p$-disjointly homogeneous) if two arbitrary normalized disjoint sequences from $E$ contain equivalent in $E$ subsequences (resp. every normalized disjoint sequence contains a subsequence equivalent in $E$ to the unit vector basis of $l_p$). Answering a question raised in the paper \cite{FHSTT}, for each $1<p<\infty$, we construct a reflexive $p$-disjointly homogeneous rearrangement invariant space on $[0,1]$ whose dual is not  disjointly homogeneous. Employing methods from interpolation theory, we provide new examples of disjointly homogeneous rearrangement invariant spaces; in particular, we show that there is a Tsirelson type disjointly homogeneous rearrangement invariant space, which contains no subspace isomorphic to $l_p$, $1\le p<\infty$, or $c_0$.
\end{abstract}

\footnotetext[1]{2010 {\it Mathematics Subject Classification}:  46E30, 46B70, 46B42.}
\footnotetext[2]{\textit{Key words and phrases}: rearrangement invariant space, disjointly homogeneous lattice, Lions-Peetre interpolation spaces, basis, block basis, dual space, complemented subspace}

\section{\protect \medskip Introduction}

A Banach lattice $E$ is called {\it disjointly homogeneous} (shortly DH) if two arbitrary normalized disjoint sequences in $E$ contain equivalent subsequences. In particular, given $1\le p\le\infty$, a Banach lattice $E$ is  {\it $p$-disjointly homogeneous} (shortly p-DH) if each  normalized disjoint sequence in $E$ has a subsequence equivalent to the unit vector basis of $l_p$ ($c_0$ when $p=\infty$). These notions were first introduced in \cite{FTT-09} and proved to be very useful in studying the general problem of identifying Banach lattices $E$ such that the ideals of strictly singular and compact operators bounded in $E$ coincide \cite{bib:04} (see also survey \cite{FHT-survey} and references therein). Results obtained there can be treated as a continuation and development of a classical theorem of V.~D.~Milman \cite{Mil-70} which states that every strictly singular operator in $L_p(\mu)$ has compact square. This is a motivation to find out how large is the class of DH Banach lattices. As it is shown in the above cited papers, it contains $L_p(\mu)$-spaces, 
$1 \le p \le\infty$, Lorentz function spaces $L_{q,p}$ and $\Lambda(W, p)$, a certain class of Orlicz function spaces, Tsirelson space and some others. 

The next step in the case of rearrangement invariant (in short, r.i.) function spaces was undertaken in the paper \cite{A15}. By the complex method of interpolation, it was proved there that for every $1 \le p \le\infty$ and any 
increasing concave function $\varphi$ on $[0,1]$, which is not equivalent to neither $1$ nor $t$, there exists a $p$-DH r.i. space on $[0,1]$ with the fundamental 
function $\varphi$ (see definitions in the next section). Observe that there is the only r.i. space on $[0,1]$, $L_\infty$ (resp. $L_1$), having the fundamental function equivalent to $1$ (resp. $t$). 
Moreover, in \cite[Theorem~4]{A15} it is obtained the following sharp version of the classical Levy's result \cite{Levy}
for sequences of pairwise disjoint functions. If $X_0$ and $X_1$ are r.i. spaces such that $X_0$ is strictly embedded into $X_1$,
then every sequence $\{x_n\}_{n=1}^\infty$ of normalized pairwise disjoint functions from the real interpolation space $(X_0,X_1)_{\theta,p}$, $0<\theta<1$, $1\le p<\infty$,
$\|x_n\|_{(X_0,X_1)_{\theta,p}}=1$, $n=1,2,\dots$, contains a subsequence equivalent to the unit vector basis of $l_p$. 

Here, we continue the above direction of research considering a special case of the real method of interpolation, which was introduced and studied by Lions and Peetre 
\cite{LP-64} (see also \cite[2g]{LTbook2}). While parameters of the functors $(\cdot,\cdot)_{\theta,p}$, $0<\theta<1$, $1\le p<\infty$, are only weighted $l_p$-spaces, the 
interpolation spaces from \cite{LP-64} are generated by arbitrary Banach spaces having a normalized 1-unconditional basis $\{e_n\}$. It turns out that in this case there is still a direct link between some properties of block bases of $\{e_n\}$ and of sequences of pairwise disjoint functions from the respective interpolation space,
which allows to construct r.i. spaces with a certain prescribed lattice structure.
This applies not only to the equivalence of subsequences but also to their complementability. As was shown in the paper \cite{FHSTT}, DH properties of
Banach lattices are closely connected with the following concept. A Banach lattice $E$ is called {\it disjointly complemented (DC)}
if every disjoint sequence from $E$ contains a subsequence complemented in $E$.

The above approach based on using tools from interpolation theory allows to construct new examples of $p$-DH and DH r.i. spaces. 
In particular, we show that there is a Tsirelson type DH and DC r.i. space, which contains no subspace isomorphic to $l_p$, $1\le p<\infty$, or $c_0$ (Theorem~\ref{Th5}). Moreover, we solve the following duality problem posed in \cite[p.~5877]{FHSTT}: Is there a reflexive $p$-DH r.i. space on $[0,1]$ whose dual is not DH 
(see also Question~3 in the survey \cite{FHT-survey})? Answering this question, we construct such a space in Theorem~\ref{main}. We believe that the approach developed here is interesting in its own and may be useful in solving other problems related to the study of lattice properties of r.i. spaces.

In the concluding part of the paper, we show that the existence of sequences of
equimeasurable pairwise disjoint functions in the case of infinite measure leads to the essential difference
of DH and DC properties of r.i. spaces on $(0,\infty)$ and $[0,1]$. As was shown in \cite[Theorem~5.2]{FHSTT}, for each $1<p<\infty$ there is a $p$-DH Orlicz space on $(0,\infty)$ whose dual is not DH. In fact, by using known results on subspaces generated  by translations in r.i. spaces due to Hernandez and 
Semenov \cite{HS}, we are able to give a characterization of $L_p$-spaces via DH (resp. DC) property (Theorem~\ref{main2}). Namely, 
we show that if $X$ is a reflexive r.i. space 
on $(0,\infty)$ such that $X\ne L_p$ for every $1<p<\infty$, then at least one of the spaces $X$ or $X^*$ is not DH.

\section{\protect Preliminaries}\label{Prel}

\subsection{\protect Banach lattices and rearrangement invariant spaces}
Let $I = [0, 1]$ or $[0, \infty)$. A {\it Banach function  lattice} $E = (E, \|\cdot\|)$ on $I$ is a Banach space of real--valued Lebesgue measurable functions (of equivalence classes) defined 
on $I$, which satisfies the so--called ideal property: if $f$ is a measurable function, $|f| \leq |g|$ almost everywhere (a.e.) with respect to the Lebesgue measure on $I$ and $g \in E$, then $ f\in E$ and $\|f\|_E \leq \|g\|_E$.  

If $E$ is a Banach lattice on $I$, then the {\it K{\"o}the dual space} (or {\it associated space}) $E^{\prime}$ consists of all real--valued  measurable functions $f$ such that 
\begin{equation} \label{dual}
\|f\|_{E^\prime} := \sup_{g \in E, \, \|g\|_{E} \leq 1} \int_{I} |f(x) g(x) | \, dx<\infty.
\end{equation}
The K{\"o}the dual $E^{\prime}$ is a Banach lattice. Moreover, $E \subset E^{\prime \prime}$ and the equality $E = E^{\prime \prime}$ holds isometrically if and only if the norm in $E$ has the {\it Fatou property}, meaning that the conditions $0 \leq f_{n} \nearrow f$ a.e. on $I$ and $\sup_{n \in {\bf N}} \|f_{n}\|_E < \infty$ imply that $f \in E$ and $\|f_{n}\|_E \nearrow \|f\|_E$. For a separable Banach lattice $E$ the K{\"o}the dual $E^{\prime}$ and the (Banach) dual space $E^{*}$ coincide. Moreover, a Banach lattice $E$ with the Fatou property is reflexive if and only if both $E$ and its K\"othe dual $E^{\prime}$ are separable.

Let $E$ be a Banach lattice on $I$. A function $f\in E$ is said to have an {\it order continuous norm} in $E$ if for any
decreasing sequence of Lebesgue measurable sets $B_{n} \subset I $ with $m(\cap_{n=1}^{\infty} B_n) = 0$, where 
$m$ is the Lebesgue measure, we have $\|f \chi_{B_{n}} \|_E \rightarrow 0$ as $n \rightarrow \infty$. The set of all functions 
in $E$ with order continuous norm is denoted by $E_0$. A Banach lattice $E$ on $I$ is separable if and only if $E_0 = E$. 

A Banach lattice $X$ on $I$ is said to be a {\it rearrangement invariant (in short, r.i.)} (or {\it symmetric}) space  
if from the conditions: functions $x(t)$ and $y(t)$ are {\it equimeasurable}, i.e.,
$$
m\{t\in I:\,|x(t)|>\tau\}=m\{t\in I:\,|y(t)|>\tau\}\;\;\mbox{for all}\;\;\tau>0,$$
and $y\in X$ it follows $x\in X$ and $\|x\|_X=\|y\|_X$. 

In particular, every measurable on $I$ function $x(t)$ is equimeasurable with the non-increasing, right-continuous rearrangement of $|x(t)|$ given by
$$
x^*(t):=\inf\{~\tau>0:\,m\{s\in I:\,|x(s)|>\tau\}\le t~\},\quad t>0.$$

If $X$ is a r.i. space on $I$, then the K{\"o}the dual $X^{\prime}$ is also r.i. In what follows, as in \cite{LTbook2}, we suppose that every r.i. space is either separable or maximal, i.e., $X=X^{\prime\prime}$.

The {\it fundamental function} $\phi_X$ of a r.i. space $X$ is defined as 
\begin{equation*}
\phi_X(t):=\|\chi_{[0,t]}\|_X, ~t>0,
\end{equation*}
where $\chi_B$, throughout, will denote the characteristic function of a set $B.$ The function $\phi_X$ is quasi-concave, that is, it is nonnegative and increases, $\phi_X(0)=0$, and the function $\phi_X(t)/t$ decreases.  

For any r.i. space  $X$ on $[0,1]$ we have $L_\infty [0,1]\subseteq X\subseteq L_1[0,1]$. In the case when $X\ne L_\infty[0,1]$ the space $X_0$ is r.i. and it coincides with the closure of $L_\infty$ in $X$ (the {\it separable part} of $X$). Next, we will repeatedly use the fact that the conditional expectation generated by a $\sigma$-algebra of measurable subsets of $[0,1]$ is a projection of norm 1 in every r.i. space on $[0,1]$ \cite[Theorem~2.a.4]{LTbook2}.

An important example of r.i. spaces are the Orlicz spaces. Let $\Phi$ be an increasing convex function on $[0, \infty)$ such that $\Phi (0) = 0$. Denote by $L_{\Phi}(I)$ the {\it Orlicz space}  
on $I$ (see e.g. \cite{KR61}) endowed with the Luxemburg--Nakano norm
$$
\| f \|_{L_{\Phi}} = \inf \{\lambda > 0 \colon \int_I \Phi(|f(x)|/\lambda) \, dx \leq 1 \}.
$$
In particular, if $\Phi(u)=u^p$, $1\le p<\infty$, we obtain $L_p(I)$. 

Similarly, one can define Banach lattices and r.i. sequence spaces (i.e., on $I = \mathbb N$ with the counting measure) and all the above notions. 

For general properties of Banach lattices and r.i. spaces we refer to the books \cite{KA77}, \cite{AB85}, \cite{LTbook2}, \cite{KPS}, \cite{BS88} and  \cite{A-17}. 

\subsection{\protect \medskip Interpolation spaces}

Recall that a pair $(A_0,A_1)$ of Banach spaces is called a {\it
Banach couple} if $A_0$ and $A_1$ are both linearly and continuously
embedded in some Hausdorff linear topological vector space. In particular, arbitrary Banach lattices $E_0$ and $E_1$ on $I$ form a Banach couple because of every such a lattice is continuously embedded into the space of all measurable a.e. finite functions on $I$ equipped with the convergence in measure on the sets of finite measure.

A Banach space $A$ is called {\it interpolation} with respect to the couple $(A_0,A_1)$ if $A_0 \cap A_1 \subset A \subset A_0+A_1$ and each linear operator $T:\,A_0+A_1\to A_0+A_1$, which is bounded in $A_0$ and in $A_1$, is bounded in $A$.  

For a Banach couple $(A_0,A_1)$ we can define {\it the intersection}
$A_0\cap A_1$ and {\it the sum} $A_0+A_1$ as the Banach spaces with
the natural norms: $\|a\|_{A_0\cap
A_1}=\max\left\{\|a\|_{A_0}\,,\,\|a\|_{A_1}\right\}$ and
$\|a\|_{A_0+A_1}=k(1,1,a;A_0,A_1),$ where $k(\alpha,\beta,a;A_0,A_1)$ is the {\it Peetre ${\mathcal K}$--functional}, i.e.,
$$
k(\alpha,\beta,a;A_0,A_1):=\inf\{\alpha\|a_0\|_{A_0}+\beta\|a_1\|_{A_1}:\,a=a_0+a_1,a_i\in{A_i},
i=0,1\}$$ 
for any $a\in {A_0+A_1}$ and $\alpha,\beta>0.$

Let $E$ be a Banach space with an 1-unconditional normalized basis $\{e_n\}_{n=1}^\infty$ and let $(\alpha_n)_{n=1}^\infty$ and $(\beta_n)_{n=1}^\infty$ be two sequences of positive numbers such that
$$
\sum_{n=1}^\infty \min(\alpha_n,\beta_n)<\infty.$$
Following \cite{LP-64} (see also \cite[Definition~2.g.3]{LTbook2}), for every Banach couple $(A_0,A_1)$ we define the {\it Lions-Peetre space} $W_E^K(A_0,A_1,(\alpha_n),(\beta_n))$ as the set of all elements $a\in A_0+A_1$, for which the series
$$
\sum_{n=1}^\infty k(\alpha_n,\beta_n,a;A_0,A_1)e_n$$
converges in $E$, and we set
$$
\|a\|_{W_E^K}:=\Big\|\sum_{n=1}^\infty k(\alpha_n,\beta_n,a;A_0,A_1)e_n\Big\|_E.$$

It is known \cite[Proposition~2.g.4]{LTbook2} that $W_E^K(A_0,A_1,(\alpha_n),(\beta_n))$ is an interpolation Banach space with respect to the couple $(A_0,A_1)$. In particular, the following continuous embeddings hold:
$$A_0\cap A_1\subset W_E^K(A_0,A_1,(\alpha_n),(\beta_n))\subset A_0+A_1.$$ 
We shall concern with the case when $\alpha_n=m_n^{-1}$ and $\beta_n=m_n$, $n=1,2,\dots$, where $(m_n)_{n=1}^\infty$ is any fixed increasing sequence such that $m_1\ge 2$ and
\begin{equation}\label{EQ2}
m_n^{-1}\sum_{i=1}^{n-1} m_i+m_n\sum_{i=n+1}^{\infty} m_i^{-1}<2^{-n-1},\;\;n=1,2,\dots
\end{equation}
(by convention, $\sum_{i=1}^{0}=0)$.

Let $A$ be a Banach space, $a_n\in A$, $n=1,2,\dots$ We shall denote by $[a_n]$ the closed linear span of a sequence $\{a_n\}$ in $A$. This sequence will be called {\it complemented} if $[a_n]$ is a complemented subspace in $A$. Moreover, if $A^*$ is the dual space for $A$, the value of a functional $a^*\in A^*$ at an element $a\in A$ will be denoted by $\langle a,a^*\rangle$. In particular, if $E'$ is the K{\"o}the dual to a Banach function lattice $E$ on $I$ (resp. Banach sequence lattice $E$), $x(t)\in E$, $y(t)\in E'$ (resp. $x=(x_k)_{k=1}^\infty\in E$, $y=(y_k)_{k=1}^\infty\in E'$), we have $\langle x,y\rangle=\int_I x(t)y(t)\,dt$ (resp. $\langle x,y\rangle=\sum_{k=1}^\infty x_ky_k$). Finally, the notation   $F\asymp G$ will mean that there exist constants $C>0$ and $c>0$ not depending on the arguments of the expressions $F$ and $G$ such that $c{\cdot}F\le G\le C{\cdot}F$.

\section{Lattice properties of interpolation spaces}\label{LP}

In what follows we assume that $X_0$ and $X_1$ are r.i. spaces, 
$X_0\subset X_1$ and this embedding is strict, that is, for each sequence $\{x_n\}_{n=1}^\infty\subset X_0$ such that $\sup_{n=1,2,\dots}\|x_n\|_{X_0}<\infty$ and $m( \supp\,x_n)\to 0$ we have $\|x_n\|_{X_1}\to 0$ as $n\to\infty$. 

Additionally, without loss of generality, we shall assume that 
\begin{equation}\label{EQ0}
\|x\|_{X_1}\le \|x\|_{X_0}\;\;\mbox{for all}\;\;x\in X_0
\end{equation}
and
\begin{equation}\label{EQ1}
\phi_{X_0}(1)=\phi_{X_1}(1)=1.
\end{equation}


In the next two propositions the spaces $X_0$, $X_1$ and $E$ will be fixed and so, for brevity, we set $W:=W_{E}^K(X_0,X_1)$.
In the proof of the first of them we make use of an idea from the proof of Proposition~3.b.4 in \cite{LTbook1}.

\begin{propos}\label{Th0}
There exists a sequence $(A_n)_{n=1}^{\infty}$ of pairwise disjoint measurable subsets of $[0,1]$ such that the sequence
$\left\{\frac{\chi_{A_n}}{\|\chi_{A_n}\|_{W}}\right\}_{n=1}^\infty$
is equivalent to the basis $\{e_n\}_{n=1}^{\infty}$ of the space $E$.
 \end{propos}
\begin{proof}
For each $n\in\mathbb{N}$ we define on $X_1$ the norm $\|\cdot\|_n$ by
$$
\|x\|_n:=k(x,m_n^{-1},m_n).$$
We claim that for any measurable set $A\subset [0,1]$, with $m(A)=t$, it holds
\begin{equation}\label{EQ3}
\|\chi_A\|_n=\min\{m_n^{-1}\phi_{X_0}(t),m_n\phi_{X_1}(t)\}.
\end{equation}

Since the inequality $\le $ is obvious, it suffices to prove the opposite one. Recall that the expectation operator corresponding to the partition $\{A,[0,1]\setminus A\}$ of $[0,1]$, i.e., the operator
$$
S_Ax(t):=\frac{1}{m(A)}\int_A x(s)\,ds\cdot \chi_A(t)+\frac{1}{m([0,1]\setminus A)}\int_{[0,1]\setminus A} x(s)\,ds\cdot \chi_{[0,1]\setminus A}(t)$$
is bounded in every r.i. space with the norm 1 (see Section~\ref{Prel} or \cite[Theorem~2.a.4]{LTbook2}).
Therefore, the value of norm $\|\chi_A\|_n$ can be computed by using only  decompositions of $\chi_A$ of the form
$$
\chi_A=\alpha \chi_A+ (1-\alpha)\chi_A,\;\;0\le\alpha\le 1.$$
Thus,
$$
\|\chi_A\|_n=\inf_{0\le\alpha\le 1}(\alpha m_n^{-1}\|\chi_A\|_{X_0}+(1-\alpha) m_n\|\chi_A\|_{X_1})\ge \min\{m_n^{-1}\phi_{X_0}(t),m_n\phi_{X_1}(t)\},$$
and therefore \eqref{EQ3} follows.

Further, since the embedding $X_0\subset X_1$ is strict,
we have
$$
\lim_{t\to 0}\frac{\phi_{X_1}(t)}{\phi_{X_0}(t)}=0.$$
Therefore, due to the continuity of the fundamental functions $\phi_{X_0}$ and $\phi_{X_1}$ and equation \eqref{EQ1}, for each $n\in\mathbb{N}$ we can find $t_n\in (0,1)$ such that
$$
\frac{\phi_{X_1}(t_n)}{\phi_{X_0}(t_n)}=m_n^{-2}.$$
Observe that $t_n\le m_n^{-2}$. Indeed, assuming that $t_n> m_n^{-2}$, by the quasi-concavity of the function $\phi_{X_1}(t)$, we infer
$$
\phi_{X_1}(t_n)\ge t_n\phi_{X_1}(1)>m_n^{-2}\phi_{X_0}(1)\ge m_n^{-2}\phi_{X_0}(t_n),$$
which contradicts the choice of $t_n$. 
Hence, from the inequality $m_n\ge 2$ and \eqref{EQ2} it follows
$$
\sum_{n=1}^\infty t_n\le \sum_{n=1}^\infty m_n^{-2}<1,$$
and so we can fix pairwise disjoint measurable subsets
$A_k\subset [0,1]$, $m(A_k)=t_k$, $k=1,2,\dots$, such that (see \eqref{EQ3} and \eqref{EQ2})
$$
\sup_{k=1,2,\dots}\|\chi_{A_n}\|_k=\|\chi_{A_n}\|_n\;\;\mbox{for all}\;\;n,k\in\mathbb{N}.$$
Show that the norms $\|\chi_{A_n}\|_k$, $k\ne n$, are negligible comparatively with the norm $\|\chi_{A_n}\|_n$.

First, if $1\le k<n$, then by \eqref{EQ3} we have
\begin{equation}\label{EQ4}
\|\chi_{A_n}\|_k\le m_k\phi_{X_1}(t_n)=\frac{m_k}{m_n}m_n\phi_{X_1}(t_n)=\frac{m_k}{m_n}\|\chi_{A_n}\|_n.
\end{equation}
Similarly, in the case when $k>n$ we obtain
\begin{equation}\label{EQ5}
\|\chi_{A_n}\|_k\le m_k^{-1}\phi_{X_0}(t_n)=\frac{m_n}{m_k}\|\chi_{A_n}\|_n.
\end{equation}
Combining these estimates with \eqref{EQ2}, we deduce 
$$
\sum_{k\ne n}\|\chi_{A_n}\|_k\le\Big(m_n^{-1}\sum_{k=1}^{n-1} m_k+m_n\sum_{k=n+1}^{\infty} m_k^{-1}\Big)\|\chi_{A_n}\|_n\le 2^{-n-1}\|\chi_{A_n}\|_n.$$
Since the basis $\{e_n\}_{n=1}^\infty$ is normalized and 1-unconditional, the latter estimate with the definition of the norm in $W$ guarantee that for every $n=1,2,\dots$ 
\begin{equation}\label{EQ6}
\|\chi_{A_n}\|_n\le \|\chi_{A_n}\|_W\le (1+2^{-n-1})\|\chi_{A_n}\|_n.
\end{equation}

Setting $u_j:=\frac{\chi_{A_j}}{\|\chi_{A_j}\|_j}$, $j=1,2,\dots$, in view of \eqref{EQ4}, \eqref{EQ5} and the left-hand side of \eqref{EQ6} we obtain
$$
\|u_j\|_n\le\min\left(\frac{m_j}{m_n},\frac{m_n}{m_j}\right),\;\;j\ne n.$$
This and inequality \eqref{EQ2} imply
$$
\sum_{j\ne n}\|u_j\|_n\le m_n^{-1}\sum_{i=1}^{n-1} m_i+m_n\sum_{i=n+1}^{\infty} m_i^{-1}<2^{-n-1},$$
which yields 
$$
\Big\|\sum_{j=1}^\infty c_ju_j\Big\|_n-|c_n|=\Big\|\sum_{j=1}^\infty c_ju_j\Big\|_n-\|c_nu_n\|_n\le \sum_{j\ne n}|c_j|\|u_j\|_n\le 2^{-n-1}\sup_{j=1,2,\dots}|c_j|$$
for arbitrary $c_j\in\mathbb{R}$, $j=1,2,\dots$ Hence,
\begin{eqnarray*}
 \Big\|\sum_{j=1}^\infty c_ju_j\Big\|_W &=&\Big\|\sum_{n=1}^\infty\Big\|\sum_{j=1}^\infty c_ju_j\Big\|_n e_n\Big\|_E\\
 &\le &\Big\|\sum_{n=1}^\infty(|c_n|+2^{-n-1}\sup_{j=1,2,\dots}|c_j|)e_n\Big\|_E\\ &\le&
\Big\|\sum_{n=1}^\infty c_ne_n\Big\|_E+ \sup_{j=1,2,\dots}|c_j|\le 2
 \Big\|\sum_{n=1}^\infty c_ne_n\Big\|_E.
 \end{eqnarray*} 
 
On the other hand, since
$$
 \Big\|\sum_{j=1}^\infty c_ju_j\Big\|_n\ge |c_n|\frac{\|\chi_{A_n}\|_n}{\|\chi_{A_n}\|_W},$$
then from \eqref{EQ6} and the 1-unconditionality of $\{e_n\}_{n=1}^\infty$ it follows
$$
\Big\|\sum_{j=1}^\infty c_ju_j\Big\|_W\ge\Big\|\sum_{n=1}^\infty |c_n|\frac{\|\chi_{A_n}\|_n}{\|\chi_{A_n}\|_W}e_n\Big\|_E\ge\frac45\Big\|\sum_{n=1}^\infty c_ne_n\Big\|_E.$$
This completes the proof.
\end{proof}

Now, we prove in a sense an opposite result. We shall need some more notation. Let $B\subset \mathbb{N}$. For any $f\in E$ we put $P_Bf:=f\chi_B$. Since the basis $\{e_n\}_{n=1}^{\infty}$ is 1-unconditional, $P_B$ is a bounded projection in $E$ with norm 1. Moreover, we set 
$$
Sx:=\sum_{k=1}^\infty \|x\|_k e_k,\;\;x\in W,$$
where $\|\cdot\|_k$, $k=1,2,\dots$, are the norms introduced in the proof of Proposition~\ref{Th0}. Clearly, $Sx\in E$ for each $x\in W$ and $\|x\|_W=\|Sx\|_E$.

\begin{propos}\label{Theor1}
Every sequence $\{x_n\}_{n=1}^\infty$ of pairwise disjoint functions, $\|x_n\|_{W}=1$, contains a subsequence $\{u_i\}_{i=1}^\infty$, which is equivalent in $W$ to some block basis of $\{e_n\}_{n=1}^\infty$.
\end{propos}
\begin{proof}
Let $\delta_i>0$, $i=1,2,\dots$, and also $\delta_0:=\sum_{i=1}^\infty \delta_i<1$. Moreover, 
suppose that functions $x_k$, $k=1,2,\dots$, are pairwise disjoint, $\|x_k\|_W=1$ for each $k$. Show that there are a subsequence $\{u_i\}_{i=1}^{\infty}$ of $\{x_k\}_{k=1}^{\infty}$ and sets $B_i=\{n\in\mathbb{N}:\,l_i\le n\le m_{i}\}$, $i=1,2,\dots$, where $1=l_1\le m_1<l_2\le m_2<\dots$, such that for $f_i:=P_{B_i}Su_{i}$, $i=1,2,\dots$, we have
\begin{equation}\label{EQ13}
\|Su_i-f_i\|_E<\delta_i,\;\;i=1,2,\dots
\end{equation}

First, since $\|x_n\|_W=1$ and the basis $\{e_n\}_{n=1}^\infty$
is 1-unconditional, then $\|x_n\|_k\le 1$, or equivalently,
$$
\inf\{m_k^{-1}\|y_n\|_{X_0}+m_k\|z_n\|_{X_1}:\,y_n+z_n=x_n\}\le 1.$$
Therefore, for each $k=1,2,\dots$ there are functions $y_{n}^k\in X_0$ and $z_{n}^k\in X_1$, $n=1,2,\dots$, such that
$y_n^k+z_n^k=x_n$, ${\rm supp}\, y_{n}^k={\rm supp}\, z_{n}^k= {\rm supp}\, x_n$, $\|y_{n}^k\|_{X_0}\le m_k$ and $\|z_{n}^k\|_{X_1}\le m_k^{-1}$. This observation and inequality \eqref{EQ0} imply
\begin{equation}\label{EQ7}
\|x_n\|_{X_1}\le \|y_{n}^k\|_{X_1}+\|z_{n}^k\|_{X_1}\le m_k^{-1}+\|y_{n}^k\|_{X_1}.
\end{equation}

Further, $m(\supp\, y_{n}^k)=m(\supp\, x_n)\to 0$ as $n\to\infty$ and $\|y_{n}^k\|_{X_0}\le m_k$, $n=1,2,\dots$ Hence, because of $X_0$ is strictly embedded into $X_1$, we infer that $\|y_{n}^k\|_{X_1}\to 0$ as $n\to\infty$ for each $k=1,2,\dots$ Thus, for every $k=1,2,\dots$ there exists an increasing sequence of positive integers $\{n_k\}_{k=1}^\infty$ such that $\|y_{n_k}^k\|_{X_1}\le m_k^{-1}$. Therefore, in view of \eqref{EQ7}, we have
\begin{equation*}\label{EQ8}
\|x_{n_k}\|_{X_1}\le 2m_k^{-1},\;\;k=1,2,\dots
\end{equation*}
Moreover, by the definition of the norm $\|\cdot\|_j$, 
$$
\|x_{n_k}\|_{j}\le m_j\|x_{n_k}\|_{X_1},\;\;j=1,2,\dots$$
Combining the latter inequalities, we deduce
$$
\|x_{n_k}\|_{j}\le \frac{2m_j}{m_k},\;\;k,j=1,2,\dots$$
and therefore from \eqref{EQ2} it follows 
\begin{equation}\label{EQ8.1}
\Big\|\sum_{j=1}^{k-1}\|x_{n_k}\|_{j}e_j\Big\|_E\le \sum_{j=1}^{k-1}\|x_{n_k}\|_{j}\le\frac{2}{m_k}\sum_{j=1}^{k-1}m_j\le 2^{-k},\;\;k=1,2,\dots
\end{equation}

We set $l_1:=1$ and $u_1:=x_1$. Since $u_1\in W$, we can find a positive integer $m_1$ such that
$$
\Big\|\sum_{j=m_1+1}^\infty\|u_{1}\|_{j}e_j\Big\|_E< \delta_1.$$
Hence, if $B_1:=\{n\in\mathbb{N}:\,l_1\le n\le m_{1}\}$ and $f_1:=P_{B_1}Su_{1}$ we get
$$
\|Su_1-f_1\|_E<\delta_1.
$$
Next, we choose $k_2>m_1+1$ such that $2^{-k_2}<\delta_2/2$. Then, setting $l_2:=k_2-1$ and $u_2:=x_{n_{k_2}}$ we have $l_2>m_1$ and, by \eqref{EQ8.1},
$$
\Big\|\sum_{j=1}^{l_2}\|u_{2}\|_{j}e_j\Big\|_E<\frac{\delta_2}{2}.
$$
Moreover, for some $m_2\ge l_2$
$$
\Big\|\sum_{j=m_2+1}^\infty\|u_{2}\|_{j}e_j\Big\|_E< \frac{\delta_2}{2}.$$
Combining the latter inequalities, we conclude 
$$
\|Su_2-f_2\|_E<\delta_2,
$$
where $B_2:=\{n\in\mathbb{N}:\,l_2\le n\le m_{2}\}$ and $f_2:=P_{B_2}Su_{2}$. 

Proceeding in the same way, we get a subsequence $\{u_i\}_{i=1}^{\infty}$ of $\{x_k\}_{k=1}^{\infty}$ and sets $B_i=\{n\in\mathbb{N}:\,l_i\le n\le m_{i}\}$, $i=1,2,\dots$, where $1=l_1\le m_1<l_2\le m_2<\dots$ such that the elements $f_i:=P_{B_i}Su_{i}$, $i=1,2,\dots$, satisfy \eqref{EQ13}. Clearly, from \eqref{EQ13} and the fact that $\|u_i\|_W=1$, $i=1,2,\dots$, it follows
\begin{equation}\label{EQ11}
\|f_i\|_E\ge 1-\delta_i,\;\;i=1,2,\dots
\end{equation}
Prove that the constructed block basis $\{f_i\}_{i=1}^\infty$ of the basis $\{e_k\}_{k=1}^\infty$ is equivalent to the sequence $\{u_i\}_{i=1}^\infty\subset W$.

Assume that the series $\sum_{i=1}^{\infty} c_iu_{i}$
converges in $W$. Since the sequence $\{u_{i}\}_{i=1}^\infty$ consists of pairwise disjoint functions and the basis $\{e_{j}\}_{j=1}^\infty$ is 1-unconditional, from the definition of the norm in $W$ it follows 
\begin{eqnarray*}
 \Big\|\sum_{i=1}^{\infty} c_iu_{i}\Big\|_W &=& \Big\|\sum_{i=1}^{\infty} |c_i|u_{i}\Big\|_W=\Big\|\sum_{j=1}^{\infty} \Big\|\sum_{i=1}^{\infty}|c_i|u_{i}\Big\|_je_j\Big\|_E\\
 &\ge& \Big\|\sum_{k=1}^{\infty}\sum_{j=l_k}^{m_k} \Big\|\sum_{i=1}^{\infty}|c_i|u_{i}\Big\|_je_j\Big\|_E
 \ge \Big\|\sum_{k=1}^{\infty} \sum_{j=l_k}^{m_k}|c_k|\|u_{k}\|_je_j\Big\|_E\\ &=& \Big\|\sum_{k=1}^{\infty} |c_k|f_{k}\Big\|_E= \Big\|\sum_{k=1}^{\infty} c_kf_{k}\Big\|_E.
 \end{eqnarray*} 
 
 Conversely, assume that the series $\sum_{k=1}^{\infty} c_kf_{k}$ converges in $E$.  Then, by \eqref{EQ13} and \eqref{EQ11}, we have 
\begin{eqnarray*}
 \Big\|\sum_{k=1}^{\infty} c_ku_{k}\Big\|_W &=& \Big\|\sum_{j=1}^{\infty} \Big\|\sum_{k=1}^{\infty}c_ku_{k}\Big\|_je_j\Big\|_E\le \Big\|\sum_{k=1}^{\infty}|c_k|\sum_{j=1}^{\infty} \|u_{k}\|_je_j\Big\|_E \\
 &\le& 
 \Big\|\sum_{k=1}^{\infty}|c_k|\sum_{j\in B_k} \|u_{k}\|_je_j\Big\|_E+\Big\|\sum_{k=1}^{\infty}|c_k|\sum_{j\not\in B_k}\|u_k\|_je_j\Big\|_E\\
 &\le& 
 \Big\|\sum_{k=1}^{\infty}c_kf_k\Big\|_E+\max_{k=1,2,\dots}|c_k|
\sum_{k=1}^{\infty}\delta_k \le
\frac{\delta_0}{1-\delta_0} \Big\|\sum_{k=1}^{\infty}c_kf_k\Big\|_E,
 \end{eqnarray*} 
 and the proof is completed.
\end{proof}

\section{DH and DC properties of interpolation spaces and duality problem}

We start with the following definitions. 

\begin{defin}\cite{FTT-09}
A Banach lattice $E$ is {\it disjointly homogeneous} (shortly DH) if two arbitrary normalized disjoint sequences from $E$ contain equivalent subsequences.

Given $1\le p\le\infty$, a Banach lattice $E$ is called  {\it $p$-disjointly homogeneous} (shortly p-DH) if each  normalized disjoint sequence has a subsequence equivalent in $E$ to the unit vector basis of $l_p$ ($c_0$ when $p=\infty$).
\end{defin}

\begin{defin}\cite{FHSTT}
A Banach lattice $E$ is called {\it disjointly complemented} ($E\in DC$) if every disjoint sequence from $E$ has a subsequence whose span is complemented in $E$.
\end{defin}

For examples and other information related to DH, $p$-DH and DC Banach lattices and r.i. spaces see \cite{FTT-09,bib:04,FHSTT,FHT-survey,HST,A15}.

Results obtained in Section~\ref{LP} allow to get a direct link between DH and DC properties of the Lions-Peetre interpolation spaces and their parameters. Next, we consider the space $E$ as a Banach lattice ordered by the unconditional basis $\{e_n\}_{n=1}^{\infty}$, i.e., if $x=\sum_{n=1}^\infty a_ne_n$ then $x\ge 0$ iff $a_n\ge 0$ for each $n=1,2,\dots$

\begin{theorem}\label{Th3}
Let $1\le p\le\infty$. The following conditions are equivalent:

(a) $E$ is a DH (resp. $p$-DH) lattice;

(b) every two normalized block bases of the basis $\{e_n\}_{n=1}^{\infty}$ contain equivalent in $E$ subsequences (resp. every  normalized block basis of $\{e_n\}_{n=1}^{\infty}$ contains a subsequence equivalent in $E$ to the unit vector basis of $l_p$);  

(c) $W:=W_E^K(X_0,X_1)$ is a DH (resp. $p$-DH) r.i. space.
\end{theorem}
\begin{proof}
Obviously, $(a)$ implies $(b)$, while the implication $(b)\Rightarrow (c)$ is an immediate consequence of Proposition~\ref{Theor1}.

Now, suppose $W$ is a DH space. Let $\{M_i\}_{i=1}^\infty$ and $\{L_i\}_{i=1}^\infty$ be two sequences of pairwise disjoint non-empty subsets of $\mathbb{N}$, and let $u_i=\sum_{k\in M_i}a_k^ie_k$ and $v_i=\sum_{k\in L_i}b_k^ie_k$, where $a_k^i, b_k^i\in\mathbb{R}$, $\|u_i\|_E=\|v_i\|_E=1$, $i=1,2,\dots$ By Proposition~\ref{Th0}, there is a sequence $(A_i)_{i=1}^{\infty}$ of pairwise disjoint measurable subsets of $[0,1]$ such that the functions $x_i:=\frac{\chi_{A_i}}{\|\chi_{A_i}\|_W}$, $i=1,2,\dots$, form a sequence equivalent in $W$ to the basis $\{e_i\}_{i=1}^{\infty}$. Hence, the functions
$y_i:=\sum_{k\in M_i}a_k^ix_k$, $i=1,2,\dots$ (resp. $z_i:=\sum_{k\in L_i}b_k^ix_k$, $i=1,2,\dots$), are pairwise disjoint, $\|y_i\|_W\asymp \|z_i\|_W\asymp 1$, $i=1,2,\dots$ and the sequence $\{y_i\}$ (resp. $\{z_i\}$) is equivalent in $W$ to the sequence $\{u_i\}$ (resp. $\{v_i\}$). By the condition, the sequences $\{y_i\}$ and $\{z_i\}$ contain equivalent in $W$ subsequences. Clearly, the sequences $\{u_i\}$ and $\{v_i\}$ share the same property (in $E$) and, as a result, we get $(a)$. 

Since for the $p$-DH property the implication $(c)\Rightarrow (a)$ can be deduced in the same way, the proof is completed.
\end{proof}

The following result answers the question raised in the paper \cite[p.~5877]{FHSTT} (see also Question~3 in the survey \cite{FHT-survey}).

\begin{theorem}\label{main}
For every $1<p<\infty$ there exists a $p$-DH reflexive r.i. space $X_p$ on $[0,1]$ such that its dual space $X_p^*$ is not DH.
\end{theorem}
\begin{proof}
Let $X_0$ and $X_1$ be arbitrary r.i. spaces on $[0,1]$ such that $X_0$ is strictly embedded into $X_1$ (for example, we can put $X_0=L_q[0,1]$, $X_1=L_r[0,1]$ with $1<r<q<\infty$). In \cite[Theorem~6.7]{FHSTT} (see also \cite[Example~1]{JO}) it is shown that for every $1<p<\infty$ there is a $p$-DH Banach lattice $E_p$ ordered by an unconditional basis $\{e_n\}_{n=1}^\infty$, $\|e_n\|_{E_p}=1$, $n=1,2,\dots$, such that the dual space $E_p^*$ is not DH. Clearly, without loss of generality, we may assume that the basis $\{e_n\}$ is 1-unconditional (cf. \cite[p.~19]{LTbook1}). Therefore, by Theorem~\ref{Th3}, $X_p:=W_{E_p}^K(X_0,X_1)$ is a $p$-DH r.i. space. Since neither $c_0$ nor $l_1$ is lattice embeddable in $X_p$, by Lozanovsky theorem (see \cite{Loz} or \cite[Theorem~4.71]{AB85}), $X_p$ is reflexive.

As above, there exists a sequence $(A_k)_{k=1}^{\infty}$ of pairwise disjoint measurable subsets of $[0,1]$ such that the sequence $\{x_k\}_{k=1}^{\infty}$, where $x_k:=\frac{\chi_{A_k}}{\|\chi_{A_k}\|_W}$, $k=1,2,\dots$, is equivalent to the basis $\{e_k\}_{k=1}^{\infty}$. Clearly, the functions $y_k:=\frac{\|\chi_{A_k}\|_W}{m(A_k)}\chi_{A_k}$, $k=1,2,\dots$, form the biorthogonal system to $\{x_k\}_{k=1}^{\infty}$. Denoting by $e_k^*$, $k=1,2,\dots$, elements of biorthogonal system to the basis $\{e_k\}_{k=1}^\infty$, we show that for all $b_k\in\mathbb{R}$ we have
 \begin{equation}\label{EQequiv}
C^{-1}\Big\|\sum_{k=1}^\infty b_ke_k^*\Big\|_{E_p^*}\le \Big\|\sum_{k=1}^\infty b_ky_k\Big\|_{X_p^*}\le C\Big\|\sum_{k=1}^\infty b_ke_k^*\Big\|_{E_p^*},
\end{equation}
where $C$ is the equivalence constant of $\{x_k\}_{k=1}^{\infty}$ and $\{e_k\}_{k=1}^{\infty}$.

Indeed, the spaces $E_p$ and $X_p$ are separable and so their duals coincide with the K\"{o}the duals. Hence, \begin{eqnarray*}
 \Big\|\sum_{k=1}^\infty b_ky_k\Big\|_{X_p^*} &\ge& \sup\Big\{\Big\langle \sum_{k=1}^\infty a_kx_k,\sum_{k=1}^\infty b_ky_k\Big\rangle:\, \Big\|\sum_{k=1}^\infty a_kx_k\Big\|_{X_p}\le 1\Big\}\\
 &\ge& C^{-1}
 \sup\Big\{\sum_{k=1}^\infty a_kb_k:\, \Big\|\sum_{k=1}^\infty a_ke_k\Big\|_{E_p}\le 1\Big\}\\
 &=& C^{-1}\Big\|\sum_{k=1}^\infty b_ke_k^*\Big\|_{E_p^*}
 \end{eqnarray*} 
Conversely (see Section~\ref{Prel} or \cite[Theorem~2.a.4]{LTbook2}), the projection
$$
Px(t):=\sum_{k=1}^\infty \langle x,y_k\rangle x_k(t),\;\;0\le t\le 1,$$
is bounded in $X_p$ with norm 1. Therefore,
\begin{eqnarray*}
 \Big\|\sum_{k=1}^\infty b_ky_k\Big\|_{X_p^*} &=& \sup\Big\{\Big\langle x,\sum_{k=1}^\infty b_ky_k\Big\rangle:\, \|x\|_{X_p}\le 1\Big\}\\
 &\le& \sup\Big\{\Big\langle Px,\sum_{k=1}^\infty b_ky_k\Big\rangle:\, \|Px\|_{X_p}\le 1\Big\}\\
 &\le&  K
 \sup\Big\{\sum_{k=1}^\infty a_kb_k:\, \Big\|\sum_{k=1}^\infty a_ke_k\Big\|_{E_p}\le 1\Big\}\\
 &=& C\Big\|\sum_{k=1}^\infty b_ke_k^*\Big\|_{E_p^*},
 \end{eqnarray*} 
and \eqref{EQequiv} is proved. Since $E_p^*$ is not a DH lattice and the functions $y_k$, $k=1,2,\dots$, are pairwise disjoint, from \eqref{EQequiv} it follows that $X_p^*$ is also not DH. 
\end{proof}

\begin{theorem}\label{Th3.5}
The following conditions are equivalent:

(a) $E$ is a DC lattice;

(b) each block basis of the basis $\{e_n\}_{n=1}^{\infty}$ contains a complemented in $E$ subsequence;  

(c) $W:=W_E^K(X_0,X_1)$ is a DC r.i. space.
\end{theorem}

\begin{proof}
First, the implication $(a)\Rightarrow (b)$ is obvious.

$(b)\Rightarrow (c)$. 
Suppose $\{x_k\}_{k=1}^{\infty}$ is a sequence of pairwise disjoint functions, $\|x_k\|_W=1$, $k\in\mathbb{N}$. We need to show that it contains a subsequence spanning in $W$ a complemented subspace. 

Let $\delta_k>0$, $k=1,2,\dots$, be such that $\sum_{k=1}^\infty\delta_k<1$ (they will be specified later). Thanks to Proposition~\ref{Theor1} and its proof, we can assume that there are sets $B_k=\{n\in\mathbb{N}:\,l_k\le n\le m_k\}$, $k=1,2,\dots$, where $1=l_1\le m_1<l_2\le m_2<\dots$, such that $\{x_k\}_{k=1}^{\infty}$ is equivalent to the sequence $\{f_k\}_{k=1}^{\infty}$, $f_k=P_{B_k}Sx_{k}$, i.e.,
for some constant $C>0$ and all $c_k\in\mathbb{R}$
\begin{equation}\label{EQ12}
C^{-1}\Big\|\sum_{k=1}^{\infty} c_kf_{k}\Big\|_E\le\Big\| \sum_{k=1}^{\infty} c_kx_{k}\Big\|_W\le C\Big\|\sum_{k=1}^{\infty} c_kf_{k}\Big\|_E.
\end{equation}
Moreover, as above, we have
\begin{equation}\label{EQ13e}
\|Sx_k-f_k\|_E<\delta_k\;\;\mbox{and}\;\;\|f_k\|_E\ge 1-\delta_k,\;\;k=1,2,\dots
\end{equation}

Now, by \cite[Theorem~1.25]{AB85}, for every $k\in\fg$ we can find a linear operator $L_k:\,W\to E$ such that
$L_kx_{k}=Sx_k$ and $|L_kx|\le Sx$ for all $x\in W.$ Define
on $W$ the linear operator $L$ by
$$
Lx:=\sum_{i=1}^\infty P_{B_i}L_ix.$$ Since the sets $B_i$, $i=1,2,\dots$, are pairwise disjoint, then
\begin{equation}\label{eq19}
|Lx|\le\sup_{i\in\fg}|L_ix|\le Sx,\;\;x\in W,
\end{equation}
whence $L$ is bounded from $W$ into $E$. Moreover, setting $u_k:=Lx_k$, $k=1,2,\dots$, we have
$$
u_k=P_{B_k}Sx_k+P_{\mathbb{N}\setminus B_k}Lx_k=f_k+P_{\mathbb{N}\setminus B_k}Lx_k.$$
Therefore, from pointwise estimate \eqref{eq19} and the first inequality in \eqref{EQ13e} it follows
$$
\|u_k-f_k\|_E=\|P_{\mathbb{N}\setminus B_k}Lx_k\|_E\le \|P_{\mathbb{N}\setminus B_k}Sx_k\|_E  = \|Sx_k-f_k\|_E<\delta_k,\;\;k=1,2,\dots$$ 
Combining this together with the second inequality in \eqref{EQ13e}, we get
\begin{equation}\label{eq19.5}
\sum_{k=1}^\infty\frac{\|u_k-f_k\|_E}{\|f_k\|_E}\le\sum_{k=1}^\infty
\frac{\delta_k}{1-\delta_k}.
\end{equation}
Since the sequence $\{f_k\}_{k=1}^\infty$ is a block basis of $\{e_k\}_{k=1}^\infty$, by hypothesis, it contains a subsequence $\{f_{k_i}\}$ complemented in $E$. 
Therefore, if $\delta_k>0$ are sufficiently small, then, by the  principle of small perturbations \cite[Proposition~1.a.9]{LTbook1}, from \eqref{eq19.5} it follows that the sequence $\{u_{k_i}\}$ is equivalent to the sequence $\{f_{k_i}\}$, i.e.,
\begin{equation}\label{EQ15}
K^{-1}\Big\|\sum_{i=1}^{\infty} c_if_{k_i}\Big\|_E\le\Big\| \sum_{i=1}^{\infty} c_iu_{k_i}\Big\|_E\le K\Big\|\sum_{i=1}^{\infty} c_if_{k_i}\Big\|_E,
\end{equation} 
and is also complemented in $E.$ Moreover, since $E$ is a separable space, then there is a bounded projection $Q:\,E\to E$ of the form:
$$
Qx:=\sum_{i=1}^\infty \langle x,v_i\rangle u_{k_i},$$ 
where $v_i\in E'$, $\langle u_{k_i},v_i\rangle =1$ and $\langle u_{k_i},v_j\rangle =0$ if $i\ne j.$

Let us show that the linear operator
$$
Tx:=\sum_{i=1}^\infty \langle Lx,v_i \rangle x_{k_i}$$ is a bounded projection in
$W$. In fact, applying successively inequalities \eqref{EQ12}, \eqref{EQ15} and \eqref{eq19}, for any $x\in W$ we obtain
\begin{eqnarray*}
\|Tx\|_{W} &\le& C\Big\|\sum_{i=1}^\infty
\langle Lx,v_i\rangle f_{k_i}\Big\|_E\le CK\Big\|\sum_{i=1}^\infty
\langle Lx,v_i\rangle u_{k_i}\Big\|_E\\ &=&
CK\|Q(Lx)\|_E\le CK\|Q\|\|Lx\|_E\\
&\le&
CK\|Q\|\|Sx\|_E=CK\|Q\|\|x\|_{W}.
\end{eqnarray*}
Moreover, for every $i=1,2,\dots$
$$
Tx_{k_i}=\sum_{i=1}^\infty \langle Lx_{k_i},v_i\rangle x_{k_i}= \sum_{i=1}^\infty
\langle u_{k_i},v_i\rangle x_{k_i}=x_{k_i}.$$ Therefore, $[x_{k_i}]$ as the image of the bounded projection $T$ is a complemented subspace of $W$, and thus the proof of the implication $(b)\Rightarrow (c)$ is complete.

$(c)\Rightarrow (a)$. Again, by Proposition~\ref{Th0}, there is a sequence $(A_i)_{i=1}^{\infty}$ of pairwise disjoint measurable subsets of $[0,1]$ such that the sequence
$\left\{\frac{\chi_{A_i}}{\|\chi_{A_i}\|_W}\right\}_{i=1}^\infty$
is equivalent in $W$ to the basis $\{e_i\}_{i=1}^{\infty}$ of $E$. Setting $y_i:=\frac{\chi_{A_i}}{\|\chi_{A_i}\|_W}$, we see that the linear operator $H:\,E\to W$ defined by $H(e_i):=y_i$, $i=1,2,\dots$, is an isomorphic embedding of $E$ into $W$ with the image $[y_i]$.

Let $M_i$, $i=1,2,\dots$, be a sequence of pairwise disjoint subsets of $\mathbb{N}$, $M_i\ne\varnothing$, $i=1,2,\dots$ and let $w_i=\sum_{k\in M_i}a_k^ie_k$, where $a_k^i\in\mathbb{R}$, $k\in M_i$, $i=1,2,\dots$ It is sufficient to find a subsequence $\{w_{i_j}\}\subset \{w_i\}$, which is complemented in $E$.

Clearly, the functions $g_i:=\sum_{k\in M_i}a_k^iy_k$, $i=1,2,\dots$, are pairwise disjoint and the sequence $\{g_i\}$ is equivalent in $W$ to the sequence $\{w_i\}$. By hypothesis, there is a subsequence $\{g_{i_j}\}\subset \{g_i\}$, which is complemented in $W$. Let $P$ be a bounded projection in $W$ whose image is $[g_{i_j}]$. Now, the operator $Q:=H^{-1}PH$ is bounded in $E$ and it is easy to check that $Q$ is a projection whose image coincides with $[w_{i_j}]$. Thus, $[w_{i_j}]$ is a complemented subspace in $E$, and the theorem is proved.
\end{proof}

Let $1< p<\infty$. It is well known that there are Orlicz sequence spaces, which contain a block basis equivalent to the unit vector basis of $l_p$ but do not have any complemented subspace isomorphic to $l_p$ \cite[Examples~4.c.6 and~4.c.7]{LTbook1}, (see also \cite{LT2}and \cite{LT3}). Observe also that other (explicitly defined) examples of Orlicz sequence spaces without complemented copies of $l_p$ have been given in \cite[Theorem~1.6]{HR-S-89} and \cite[Theorem~3.4]{Kal-90}. Combining this together with the implication $(c)\Rightarrow (a)$ of Theorem~\ref{Th3.5}, we can construct r.i. spaces without the DC property. 

\begin{corollary}\label{cor5}
There exist Orlicz sequence spaces $l_N$ such that the r.i. space $W_{l_N}^K(X_0,X_1)$ is not DC.
\end{corollary}

Taking for $E$ a DH (resp. DC) Banach lattice ordered by a normalized 1-unconditional basis, according to Theorem~\ref{Th3} (resp. Theorem~\ref{Th3.5}), we get a DH (resp. DC) r.i. space $W_E^K(X_0,X_1)$ whenever r.i. spaces $X_0$ and $X_1$ are such that $X_0\subset X_1$ and this embedding is strict. In particular, we have

\begin{corollary}\label{lp-cor}
Let $1\le p\le\infty$. Then, $W_{l_p}^K(X_0,X_1)$ ($W_{c_0}^K(X_0,X_1)$ if $p=\infty$) is a $p$-DH and DC r.i. space.
\end{corollary}
 
Another interesting class of parameters enjoying DH and DC properties is formed by separable Banach sequence spaces having the so-called (RSP) and (LSP) properties.

Given a sequence $x=(x(k))_{k=1}^{\infty}$ of reals, the
\textit{support} of $x$ is the set $\supp\, x:=\,\{k\in{\mathbb {N}}:\,x(k)\ne 0\}.$
If $A\subset{\mathbb{N}}$ and $B\subset{\mathbb{N}}$ then the inequality $A<B$ means
that $a<b$ for arbitrary $a\in A,$ $b\in B.$ Let
$\{x_n\}_{n=1}^m$ and $\{y_n\}_{n=1}^m$ be two families of sequences. The
pair $(x_n,y_n)_{n=1}^m$ is \textit{interlaced} if supports of the sequences $x_n$, $y_n$, $n=1,2,\dots,m$, are finite and also
$$
\supp\,x_n<{\supp\,y_n}\;(1\le n\le m),\;\;\supp\,y_n
<{\supp\,x_{n+1}}\;(1\le n\le{m-1}).$$ We say that a Banach sequence
space $E$ has \textit{the right-shift property} (RSP) if there
exists $C_{RS}>0$ such that for any interlaced pair
$(x_n,y_n)_{n=1}^m$ with $||y_n||_E\le{||x_n||_E}=1$, $1\le n\le m$,
and for all $a_n\in\mathbb{R}$ we have
$$
\Big\|\sum_{n=1}^m a_n y_n\Big\|_E\,\le{\,C_{RS}\Big\|\sum_{n=1}^m
a_n x_n\Big\|_E}.$$ Analogously, $E$ has \textit{the left-shift
property} (LSP) if for some $C_{LS}>0$, any interlaced pair
$(x_n,y_n)_{n=1}^m$, $||x_n||_E\le{||y_n||_E}=1$  $(1\le n\le m)$
and all $a_n\in\mathbb{R}$
$$
\Big\|\sum_{n=1}^m a_n x_n\Big\|_E\,\le{\,C_{LS}\Big\|\sum_{n=1}^m
a_n y_n\Big\|_E}.$$

\begin{theorem}\label{Th3.6}

(i) If the space $E$ has at least one of the properties (RSP) or (LSP), then $W:=W_E^K(X_0,X_1)$ is a DH r.i. space.

(ii) If $E$ has both properties (RSP) and (LSP), then $W$ is a DC r.i. space.
 \end{theorem}
\begin{proof}

(i) Assume that the space $E$ possesses the (LSP)-property. According to Theorem~\ref{Th3}, it is sufficient to check that  every two normalized block bases $\{f_{k}\}_{k=1}^{\infty}$ and $\{g_{k}\}_{k=1}^{\infty}$ of $\{e_n\}_{n=1}^{\infty}$ contain subsequences equivalent in $E$. Since  supports of the elements $f_k$ and $g_k$ are finite, we can find subsequences $\{f_{i}'\}_{i=1}^{\infty}\subset \{f_{k}\}$ and $\{g_{i}'\}_{i=1}^{\infty}\subset \{g_{k}\}$ such that the pair $(f_i',g_i')_{i=1}^m$ is interlaced for each $m=1,2,\dots$.
Therefore, since $\|f_{i}'\|_E=\|g_{i}'\|_E=1$, $i\in\mathbb{N}$, then applying the (LSP)-property of $E$, for all $a_i\in\mathbb{R}$ we have 
$$
\Big\|\sum_{i=1}^m a_i
f_{i}'\Big\|_E\le{C_{RS}\Big\|\sum_{i=1}^m a_i g_{i}'\Big\|_E}.$$
Hence, if the series $\sum_{i=1}^\infty a_i g_{i}'$ converges in $E$, we get
$$
\Big\|\sum_{i=1}^\infty a_i
f_{i}'\Big\|_E\le{C_{RS}\Big\|\sum_{i=1}^\infty a_i g_{i}'\Big\|_E}.$$

Similarly, there are further subsequences
$\{f_{j}''\}_{j=1}^{\infty}\subset \{f_{j}'\}_{j=1}^{\infty}$ and $\{g_{j}''\}_{j=1}^{\infty}\subset \{g_{j}'\}_{j=1}^{\infty}$ such that
the pair $(g_i'',f_i'')_{i=1}^m$ is interlaced for each $m=1,2,\dots$.
Hence, in the same way as above, we infer
$$
\Big\|\sum_{i=1}^\infty a_i
g_{i}''\Big\|_E\le{C_{RS}\Big\|\sum_{i=1}^\infty a_i f_{i}''\Big\|_E}.$$
Thus, $\{f_{j}''\}_{j=1}^{\infty}$ and $\{g_{j}''\}_{j=1}^{\infty}$ are equivalent in $E$. In the case when $E$ has (RSP), the proof follows by the same lines.

(ii). If the space $E$ has both the properties (RSP) and (LSP), then from \cite[Lemma~2.6 and subsequent remark]{Kal} it follows that each block basis of $\{e_n\}_{n=1}^\infty$ is complemented in $E$. Therefore, applying Theorem~\ref{Th3.5}, we arrive at the desired result.
\end{proof}

It is well known that Tsirelson's space $T$ possesses the properties (RSP) and (LSP) (see \cite{CaSh} and \cite{FJ}). Since $T$ does not contain subspaces isomorphic to $l_p$, $1\le p<\infty$ or $c_0$, from 
Proposition~\ref{Theor1} and Theorem~\ref{Th3.6} we obtain

\begin{corollary}\label{Th4}
$W_T^K(X_0,X_1)$ is a DH and DC r.i. space, in which no sequence of pairwise disjoint functions is equivalent to the unit vector basis of $l_p$, $1\le p<\infty$ or $c_0$. In particular, $W_T^K(X_0,X_1)$ is not a $p$-DH space for each $1\le p\le\infty$.
\end{corollary}

A more special choice of the spaces $X_0$ and $X_1$ allows to construct a DH and DC r.i. space, which is even not isomorphic to any $p$-DH space (cf. \cite{T}).

\begin{corollary}\label{Th5}
Let $X_0$ and $X_1$ be r.i. spaces such that $X_0\subset X_1$ and this embedding is strict. Moreover, suppose that $X_1\not\supset (Exp L^2)_0$, where $(Exp L^2)_0$ is the separable part of the Orlicz space $Exp L^2$ generated by the function $e^{u^2}-1$. Then, $W_T^K(X_0,X_1)$ is a DH and DC r.i. space, which contains no subspace isomorphic to  $l_p$, $1\le p<\infty$, or $c_0$. In particular, the space $W_T^K(X_0,X_1)$ is not isomorphic to a $p$-DH Banach lattice for every $p\in [1,\infty]$.
\end{corollary}
\begin{proof}
Put $W:=W_T^K(X_0,X_1)$.
At first, we observe that the norms of $W$ and $L_1$ are not equivalent on any infinite dimensional subspace of $W$. Indeed, since $W\subset X_1$, then from the condition it follows that $W\not\supset (Exp L^2)_0$. Therefore, assuming the contrary, by \cite[Theorem~2]{AHS}, we can find a sequence $\{x_k\}$ of pairwise disjoint functions from $W$ such that the norms of $W$ and $L_1$ are equivalent on the subspace $[x_k]$. Since $W$ is a r.i. space, $W\ne L_1$, this is a contradiction with  a result from \cite{Nov} (see also \cite[Corollary~3]{A99}). 

Now, let $Y$ be an arbitrary infinite dimensional subspace of $W$.  According to the preceding observation, by Kadec-Pelczynski alternative \cite{KadPel}, $Y$ contains a sequence, which is equivalent in $W$ to some sequence of pairwise disjoint functions. Clearly (see Corollary~\ref{Th4}), the latter sequence cannot be equivalent to the unit vector basis of $l_p$, $1\le p<\infty$, or $c_0$. Hence, $Y$ is not isomorphic to $l_p$ for any $1\le p<\infty$ or $c_0$. As an immediate consequence, we deduce that $W$ is not isomorphic to a $p$-DH Banach lattice for each $p\in [1,\infty]$.
      
\end{proof}

\section{\protect \medskip A characterization of $L_p$-spaces on $(0,\infty)$ via DH properties}

In this concluding part of the paper, we turn to the case of r.i. spaces on $(0,\infty)$. Let $1<p<\infty$. In Theorem~5.2 of the paper \cite{FHSTT}, one can find an example of a $p$-DH Orlicz space $L_M(0,\infty)$ whose dual is not DH. In fact, known results on subspaces generated  by translations in r.i. spaces due to Hernandez and Semenov \cite{HS} easily imply the following characterization of $L_p$-spaces. 
\begin{theorem}\label{main2}
Let $X$ be a reflexive r.i. space on $(0,\infty)$. The following conditions are equivalent:

(a) $X$ and $X^*$ are DH;

(b) $X$ and $X^*$ are DC;

(c) $X=L_p(0,\infty)$ for some $1<p<\infty$.
\end{theorem}
\begin{proof}
Since $X$ is reflexive, the implication $(a)\Rightarrow (b)$ is an immediate consequence of \cite[Proposition~4.10]{FHSTT}. 

$(b)\Rightarrow (c)$. Let $a\in X$, $a\ne 0$, ${\rm supp}\,a\subset [0,1)$. Define the sequence of translations of the function $a$: $\tau_n a(t)=a(t-n+1)$ if $t\in [n-1,n)$ and $\tau_n a(t)=0$ if $t\not\in [n-1,n)$, $n=1,2,\dots$ We prove that the sequence $\{\tau_{n}a\}$ is complemented in $X$.

By hypothesis, there is  subsequence $\{\tau_{n_k}a\}\subset \{\tau_{n}a\}$, which is complemented in $X$. Since $X$ is separable, the dual space $E^*$ coincides with the K\"{o}the dual. Hence, there is a sequence $\{b_k\}\subset E^*$ such that ${\rm supp}\,b_k\subset [n_k-1,n_k)$, $\int_{n_k-1}^{n_k} b_k \tau_{n_k}a\,ds=1$ and the projection
$$
Px(t):=\sum_{k=1}^\infty\int_{n_k-1}^{n_k} b_k(s)x(s)\,ds\cdot \tau_{n_k}a(t)$$
is bounded in $X$. Let us define the isometric embedding $R:\,X\to X$ by
$$
Rx(s):=\sum_{k=1}^\infty x(s+k-n_k)\chi_{[n_k-1,n_k)}(s),\;\;s>0.$$
Moreover, we put $c_k(s):=b_k(s-k+n_k)$, $k=1,2,\dots$, and
$$
Qx(t):=\sum_{k=1}^\infty\int_{k-1}^{k} c_k(s)x(s)\,ds\cdot \tau_{k}a(t).$$
Clearly, ${\rm supp}\,c_k\subset [k-1,k]$, $k=1,2,\dots$, and
$$
Qx(t)=\sum_{k=1}^\infty\int_{n_k-1}^{n_k} b_k(s)Rx(s)\,ds\cdot \tau_{k}a(t).$$
Therefore, the functions $Qx$ and $PRx$ are equimeasurable and so
$$
\|Qx\|_X=\|PRx\|_X\le \|P\|\|Rx\|_X=\|P\|\|x\|_X.$$
We have also $R\tau_{k}a=\tau_{n_k}a$ for all $k=1,2,\dots$, whence $Q\tau_{k}a=\tau_{k}a$, $k=1,2,\dots$ Thus, $Q$ is a bounded projection whose image coincides with the subspace $[\tau_{k}a]$, and our claim is proved. 

Similarly, we can prove that the sequence $\{\tau_{k}a^*\}$ is complemented in $X^*$ for each $a^*\in X^*$, $a^*\ne 0$, ${\rm supp}\,a^*\subset [0,1]$. Combining this together with Theorem~5.4 from the paper \cite{HS} (see also \cite[Theorem~7]{HSe}), we arrive to $(c)$. 

Since the implication $(c)\Rightarrow (a)$ is obvious, the proof is completed.
\end{proof}

\begin{corollary}\label{last}
If $X$ is a reflexive r.i. space on $(0,\infty)$ such that $X\ne L_p$ for every $1<p<\infty$, then at least one of the spaces $X$ or $X^*$ is not DH. 
\end{corollary}

\begin{remark}
An inspection of the proof of Theorem~\ref{main2} shows that DH (resp.  DC) properties of spaces $X$ and $X^*$ are used not to the full extent of their power. Indeed, it is sufficient that one can select "nice" subsequences only from positive integer-valued translations of functions with support in $[0,1]$. In this regard, it is worth to note that just the existence of such a class of sequences of equimeasurable pairwise disjoint functions in the case of infinite measure is a root of the difference of DH and DC properties of r.i. spaces on $(0,\infty)$ and $[0,1]$.
\end{remark}

\begin{remark}
For reflexive Banach lattices ordered by subsymmetric normalized bases  there is an analogue of Theorem~\ref{main2} (see \cite[Proposition~6.9]{FHSTT}).
\end{remark}


\end{document}